\documentclass[11pt, oneside]{article}
\usepackage{amsfonts,amssymb,amsmath,latexsym,amsthm}
\usepackage{mathrsfs,stmaryrd}
\usepackage{color}
\usepackage{hyperref}
\usepackage{enumerate,indentfirst,mathtools,verbatim,authblk}
\usepackage{tikz}
\usepackage{psfrag}
\usepackage{graphicx}
\usepackage{bm}
\usepackage{appendix}
\usepackage[rflt]{floatflt}
\usepackage{float}
\usepackage[square, comma, sort&compress, numbers]{natbib}
\numberwithin{equation}{section}
\allowdisplaybreaks

\newtheorem{theorem}{Theorem}[section]
\newtheorem{proposition}[theorem]{Proposition}
\newtheorem{lemma}[theorem]{Lemma}
\newtheorem{corollary}[theorem]{Corollary}
\theoremstyle{definition}

\theoremstyle{assumption}

\theoremstyle{remark}

\numberwithin{equation}{section}

\usepackage{amssymb,amsmath}
\begin{document}
\title{\LARGE\bf {Stabilization for a degenerate wave equation with time-varying delay in the boundary control input}}
\author{Menglan Liao\thanks{Corresponding author: Menglan Liao \newline \hspace*{6mm}{\it Email
addresses:}  liaoml@hhu.edu.cn}}
\affil{School of Mathematics, Hohai University, Nanjing, {\rm210098}, China.}
\affil{}
\renewcommand*{\Affilfont}{\small\it}
\renewcommand*{\Affilfont}{\small\it}
\date{} \maketitle
\vspace{-20pt}

{\bf Abstract:}  A  degenerate wave equation with time-varying delay in the boundary control input is considered. The well-posedness of the system is established by applying the semigroup theory. The boundary stabilization of the degenerate wave equation is concerned and the uniform exponential decay of solutions is obtained by combining the energy estimates with suitable Lyapunov functionals and an integral inequality under suitable conditions.



\thispagestyle{empty}
\section{Introduction}
In this paper, we study the well-posedness and stability for the following degenerate wave equation with time-varying delay in the boundary control input
\begin{equation}\label{1.1}
\begin{cases}
u_{tt}-(a(x)u_x)_x=0,\quad&(t,x)\in (0,T)\times(0,1),\\
\mu_1u_t(t,1)+\mu_2u_t(t-\tau(t),1)+u_x(t,1)+\beta u(t,1)=0,\quad&t\in(0,T),\\
u_t(t-\tau(0),1)=f_0(t-\tau(0)),\quad&t\in[0,\tau(0)],\\
\begin{cases}
u(t,0)=0\quad&\text{if }\mu_a\in[0,1),\\
\lim_{x\to0^+}a(x)u_x(t,x)=0\quad&\text{if }\mu_a\in[1,2),
\end{cases}&t\in(0,T),\\
u(0,x)=u_0(x),~u_t(0,x)=u_1(x),\quad&x\in[0,1].\\
\end{cases}
\end{equation}
The function $a\in C([0,1])\cap C^1((0,1])$ and satisfies 
\begin{equation}
\label{1.2}
\begin{cases}
(1)~a(x)>0\quad\forall x\in(0,1],\quad a(0)=0,\\
(2)~\mu_a:=\sup_{0<x\le 1}\frac{x|a'(x)|}{a(x)}<2,\\
(3)~a\in C^{[\mu_a]}([0,1]),
\end{cases}
\end{equation}
where $[\cdot]$ means the integer part. The degeneracy of problem \eqref{1.1} at $x = 0$ is measured by the parameter $\mu_a$ and  one says that problem \eqref{1.1} degenerates \emph{weakly} if $\mu_a\in [0,1)$, \emph{strongly} if $\mu_a>1$.  The time-varying delay $\tau(t)\in W^{2,\infty}([0,T])$ for any $T>0$, and there are positive constants $\tau_0$, $\tau_1$ and $d$ so that 
\begin{equation}
\label{1.3}
0<\tau_0\le \tau(t)\le \tau_1,\quad 0\le\tau'(t)\le d<1\quad \forall t>0.
\end{equation}
$\mu_1$ is a
positive real number, and $\mu_2$ is a real number and satisfy 
\begin{equation}
\label{2144}
\mu_1\ge \frac{|\mu_2|}{\sqrt{1-d}}.
\end{equation}
The initial datum $(u_0,u_1,f_0)$ belongs to a suitable space.

Degenerate equations have received a lot of attention mainly due to their accurate descriptions of several complex phenomena in numerous fields of science, particularly in biology and engineering.  Compared with  degenerate parabolic equations, the related results such as  controllability, observability and stabilization issues on degenerate hyperbolic equations are less.  The pioneering and important work on the control and stabilization of degenerate hyperbolic equations was obtained by Alabau-Boussouira, Cannarsa and Leugering \cite{A2017}(see also the special case of a power-like coefficient $a(x)=x^\alpha$ with $\alpha\in (0,2)$ studied by Zhang and Gao \cite{Z2017}). More precisely, they  proved that  $\eqref{1.1}_2$ with $\mu_1=1,~\mu_2=0$(i.e.  without time-varying delay) stabilizes exponentially the corresponding solution of the degenerate wave equation based on 
the dominant energy approach and suitable elliptic estimates.  Since then,  hyperbolic degenerate problems have received a lot of attention.  Boutaayamou, Fragnelli and  Mugnai \cite{B2023} considered the boundary controllability and the last two authors \cite{F2024} also studied the linear stabilization for a degenerate wave equation in non-divergence form with drift.   For degenerate beam problems, the controllability and stabilization results can be found in \cite{C20241,C20242,C20243}. It is worth noting that there is no delay term in all the previous papers. Until recently, Camasta, Fragnelli, Pignotti \cite{C2025}  have investigated first the well-posedness and stability for the following degenerate problems with time delay in the internal and nonlinear source
\begin{equation*}
\begin{cases}
y_{tt}(t,x)-A_i y(t,x)+k(t)BB^*y_t(t-\tau,x)=f(y(t,x)),\quad&(t,x)\in (0,\infty)\times(0,1),\\
y(0,x)=y^0(x),~y_t(0,x)=y^1(x),\quad&x\in(0,1),\\
B_iy(t,0)=0,\quad&t>0,\\
C_iy(t,1)=0,\quad&t>0,\\
B^*y_t(s,x)=g(s),\quad&s\in[-\tau,0],
\end{cases}
\end{equation*}
$i=1,2,3,4.$ $B_iy(t,0)$ and $C_iy(t,1)$  are suitable boundary conditions related to the operators $A_i$ with $i=1,2,3,4.$ The bounded linear operator $B$ that appears is defined on a real Hilbert space with adjoint $B^*$ and $f$ is a nonlinear function satisfying suitable hypotheses. By rewriting the problems in an abstract way and by using
semigroup theory and energy method, they  established the  well-posedness and stability.

It is well-known that time delay often appears in many biological, electrical engineering systems and mechanical applications.  And in many cases, the time delay may destabilize the system, see \cite{D1986,D1988,D1997} and the references therein. Therefore, it is crucial to discuss the stabilization of the wave equation under the effect of time delay in boundary or internal. It is certainly beyond
the scope of the present paper to give a comprehensive review for the PDEs with time-delay effects, the interested readers can refer to \cite{D2024,N2006,N2007,NV2009,N2011,S2021} and the references therein.  In this regard, we would like to give several references to make a comparison with this paper. Nicaise and Pignotti \cite{N2006} obtained the exponential stability for problem \eqref{1.1} with $a(x)=1$, $\beta=0$ in the case of the constant time delay(i.e. $\tau(t)=\tau>0$) in higher dimensions under the condition of $0<\mu_2<\mu_1$.  Subsequently, by introducing suitable energies and suitable Lyapunov functionals, the two authors and Fridman \cite{N2011} extended this result  to the case of the time-varying delay under the condition of $0<\mu_2<\sqrt{1-d}\mu_1$ and $\tau'(t)\le d<1$.   These results mainly concern the non-degenerate wave equations with time delay.  To our best knowledge, there is no related work that considers the case of degenerate wave equations with time-varying delay. The purpose of this paper is to explore the effect of the time-varying delay term in the boundary feedback to the stability of the degenerate wave equation. Due to the presence of the time-varying delay, we find it difficult to apply directly multiplier methods to establish integral inequalities, and then to give the exponential decay. By combining the energy estimates with suitable Lyapunov functionals and an integral inequality, we prove the uniformly exponential decay of solutions.

The paper is organized as follows. In section \ref{sec2}, we recall functional spaces and some pivotal lemmas.  In section \ref{sec3}, we combine Kato's variable norm technique and the semigroup theory to establish the well-posedness. The boundary stabilization result(see Theorem \ref{thm4.5}) for $\beta>0$ is discussed in section \ref{sec4}.

\section{Preliminaries}\label{sec2}
Following \cite{A2017}, let us denote by $V_a^1(0, 1)$ the space of all functions $u\in L^2(0, 1)$ such that
$u$ is locally absolutely continuous in $(0, 1]$, and $\sqrt{a}u_x\in L^2(0, 1).$
It is not difficult  to see that $V_a^1(0, 1)$ is a Hilbert space with the scalar product
\[\langle u,v\rangle_{1,a}=\int_0^1\Big(a(x)u'(x)v'(x)+u(x)v(x)\Big)dx\quad u,v\in V_a^1(0, 1)\]	
and associated norm
\[\|u\|_{1,a}=\Big\{\int_0^1\Big(a(x)|u'(x)|^2+|u(x)|^2\Big)dx\Big\}^{\frac{1}{2}}\quad u\in V_a^1(0, 1).\]

Denote $W_a^1(0,1)$ by the closed subspace of $V_a^1(0, 1)$ consisting of all the functions $u\in V_a^1(0, 1)$ such that $u(0)=0$ if $\mu_a\in [0,1),$ and if $\mu_a\in [1,2)$, the space $V_a^1(0, 1)$ itself. In addition, we define 
\[W_a^2(0,1):=V_a^2(0, 1)\cap W_a^1(0,1),\]
where $V_a^2(0, 1)$ is the space of all functions $u\in V_a^1(0, 1)$ satisfying $au'\in H^1(0,1)$.

\begin{proposition}\label{prop2.1}\cite[Proposition 2.5]{A2017}
Let \eqref{1.2} hold. Then the following properties hold 
\begin{enumerate}[$(1)$]
  \item For every $u\in V_a^1(0,1)$, \[\lim_{x\to 0^+}xu
  ^2(x)=0,\]
  \begin{equation}
\label{1538}
u^2(1)\le \max\Big\{2,\frac{1}{a(1)}\Big\}\|u\|_{1,a}^2.
\end{equation}
  Moreover, if $\mu_a\in [0,1)$, then $u$ is absolutely continuous in $[0,1]$.
  \item For every $u\in V_a^2(0,1)$, \[\lim_{x\to 0^+}xa(x)(u'
  (x))^2=0.\]
  For all $u\in V_a^2(0,1)$ and $\varphi\in V_a^1(0,1)$
  \[\lim_{x\to0}a(x)\varphi(x)u'(x)=0,\]
  assuming, in addition, $\varphi(0)=0$ when $\mu_a\in [0,1).$
   \item If $\mu_a\in [1,2)$, then for every $u\in V_a^2(0,1)$
   \[\lim_{x\to0}a(x)u'(x)=0.\]
\end{enumerate}
\end{proposition}

\begin{proposition}\cite[Proposition 4.3]{A2017}
Let \eqref{1.2} hold. Then 
\begin{equation}
\label{2021}
\|u\|_{L^2(0,1)}^2\le 2|u(1)|^2+C_a'\int_0^1a(x)|u'|^2dx\quad\forall u\in W_a^1(0,1),
\end{equation}
where $C_a'=\frac{1}{a(1)}\min\Big\{4,\frac{2}{2-\mu_a}\Big\}$.
\end{proposition}
\begin{proposition}\cite[Proposition 4.4]{A2017}
Let \eqref{1.2} hold and $\beta>0$. Then for every $\lambda\in \mathbb{R}$, the variational problem 
\begin{equation*}
\int_0^1 az'\varphi'dx+\beta a(1)z(1)\varphi(1)=\lambda a(1)\varphi(1)\quad\forall \varphi\in W_a^1(0,1)
\end{equation*}
admits a unique solution $z\in W_a^1(0,1)$ which satisfies the elliptic estimates 
\begin{equation}
\label{1729}
\interleave z\interleave _{1,a}^2\le \frac{a(1)}{\beta}\lambda^2,\quad \|z\|_{L^2(0,1)}^2\le \frac{a(1)}{\beta\alpha_a}\lambda^2,
\end{equation}
with \[\interleave z\interleave _{1,a}=\Big(\int_0^1 a(x)|z'(x)|^2dx+\beta a(1)z^2(1)dx\Big)^{\frac12},\quad\alpha_a=\min\Big\{\frac{1}{C_a'},\frac{\beta a(1)}{2}\Big\}>0.\] Moreover, $z\in W_a^2(0,1)$ and solves
\begin{equation}
\label{1736}
\begin{cases}
-(az')'=0,\\
z'(1)+\beta z(1)=\lambda.
\end{cases}
\end{equation}
\end{proposition}
To prove the uniformly exponential decay, we need the following well-known result.
\begin{theorem}\label{thmin}\cite[Theorem 8.1]{K1994}
Assume that $E:[0,+\infty)\to [0,+\infty)$ is a nonincreasing function and that there is a constant $M>0$ such that 
\[\int_t^\infty E(s)ds\le ME(t)\quad \forall t\in [0,+\infty).\]
Then we have 
\[E(t)\le E(0)e^{1-\frac{t}{M}}\quad \forall t\in [M,+\infty).\]
\end{theorem}

\section{Well-posedness}\label{sec3}
Let us introduce a new variable  
\begin{equation*}
w(\delta,t):=u_t(t-\delta\tau(t),1)\quad \forall (\delta,t)\in (0,1)\times (0,T),
\end{equation*}
therefore, problem \eqref{1.1} can be reformulated to the following equivalent form
\begin{equation}\label{3.2}
\begin{cases}
u_{tt}-(a(x)u_x)_x=0,\quad&(t,x)\in (0,T)\times(0,1),\\
\tau(t)w_t(\delta,t)+(1-\delta\tau'(t))w_\delta(\delta,t)=0,\quad& (\delta,t)\in (0,1)\times (0,T),\\
\mu_1u_t(t,1)+\mu_2w(1,t)+u_x(t,1)+\beta u(t,1)=0,\quad&t\in(0,T),\\
w(0,t)=u_t(t,1),\quad&t\in(0,T),\\
w(\delta,0)=f_0(-\delta\tau(0)),\quad&\delta\in(0,1),\\
\begin{cases}
u(t,0)=0\quad&\text{if }\mu_a\in[0,1),\\
\lim_{x\to0^+}a(x)u_x(t,x)=0\quad&\text{if }\mu_a\in[1,2),
\end{cases}&t\in(0,T),\\
u(0,x)=u_0(x),~u_t(0,x)=u_1(x),\quad&x\in[0,1].
\end{cases}
\end{equation}

Consider the Hilbert space 
\[\mathcal{H}=W_a^1(0,1)\times L^2(0,1) \times L^2(0,1)\]
with the scalar product
\begin{equation*}
\begin{split}
\langle(u,v,w),(\tilde{u},\tilde{v},\tilde{w})\rangle_{\mathcal{H}}&=\int_0^1 \Big(v(x)\tilde{v}(x)+a(x)u'(x)\tilde{u}'(x)\Big)dx\\
&\quad+a(1)\beta u(1)\tilde{u}(1)+\mu_1 a(1)\int_0^1w(\delta)\tilde{w}(\delta)d\delta
\end{split}
\end{equation*}
for all $(u,v,w),~(\tilde{u},\tilde{v},\tilde{w})\in \mathcal{H}$.

Define a operator $\mathcal{A}(t)$ in $\mathcal{H}$ by
\begin{equation*}
\mathcal{A}(t)(u,v,w)=\left(v,(a(x)u'(x))',\frac{\delta\tau'(t)-1}{\tau(t)}w_\delta\right)
\end{equation*}
with domain 
\begin{equation*}
\begin{split}
\mathfrak{D}(\mathcal{A}(t))=\Big\{(u,v,w)\in &W_a^2(0,1)\times W_a^1(0,1)\times H^1(0,1):~w(0):=v(1),\\
&\mu_1v(1)+\mu_2w(1)+u'(1)+\beta u(1)=0\Big\}.
\end{split}
\end{equation*}
It is noted that the domain of the operator $\mathcal{A}(t)$ is not dependent on the time $t$, in other words, 
\begin{equation*}
\mathfrak{D}(\mathcal{A}(t))=\mathfrak{D}(\mathcal{A}(0))\quad \forall t>0.
\end{equation*}

It is obvious that problem \eqref{3.2} can be transformed into the following Cauchy problem
\begin{equation}
\label{3.4}
\begin{cases}
\frac{dU(t)}{dt}=\mathcal{A}(t)U(t) \quad \forall t>0,\\
U(0)=U_0,
\end{cases}
\end{equation}
where $U(t)=(u,u_t,w)$ and $U_0=(u_0,u_1,f_0(-\delta\tau(0)))$.

In what follows,  we analyze the well-posedness of  the Cauchy problem \eqref{3.4} based on the semigroup theory \cite{K1976,K1985,P1983}. It is sufficient to prove that the triplet $\{\mathcal{A},\mathcal{H},Y\}$  determines a constant domain system for $\mathcal{A}=\{\mathcal{A}(t)|t\in [0,T]\}$ with some fixed $T>0$ and $Y=\mathcal{A}(0).$ To be a little precise, the following theorem which can be found in \cite[Theorem 1.9]{K1976} and  \cite[Theorem 2.13]{K1985} on existence and uniqueness of solutions is going to be exploited in the present paper.
\begin{theorem}
\label{thm1}
Suppose that
\begin{enumerate}[$(1)$]
  \item $Y=\mathfrak{D}(\mathcal{A}(0))$ is dense subset of $\mathcal{H}$.
  \item $\mathfrak{D}(\mathcal{A}(t))$ is independent of time $t$, that is, $\mathfrak{D}(\mathcal{A}(t))=\mathfrak{D}(\mathcal{A}(0))$ for all $t>0.$
  \item $\mathfrak{D}(\mathcal{A}(t))$ generates a strongly continuous semigroup on $\mathcal{H}$ for all $t\in [0,T]$, and the family $\mathcal{A}=\{\mathcal{A}(t)|t\in [0,T]\}$ is stable with stability constants $C$ and $m$ independent of $t$, that is, the semigroup $(S_t(s))_{s\ge 0}$ generated by $\mathcal{A}(t)$ satisfies $\|S_t(s)U\|_{\mathcal{H}}\le Ce ^{ms}\|U\|_{\mathcal{H}}$ for all $U\in \mathcal{H}$ and $s\ge 0.$
  \item $\frac{d}{dt}\mathcal{A}(t)$ belongs to $L_*^\infty([0,T],B(Y,\mathcal{H}))$, which is the space of equivalent classes of essentially bounded, strongly measurable functions from $[0,T]$ into the set $B(Y,\mathcal{H})$ of bounded linear operators from $Y$ into $\mathcal{H}$.
\end{enumerate}

Then, problem \eqref{3.4} has a unique solution $U\in C([0,T],Y)\cap C^1([0,T],\mathcal{H})$ for all initial data in $Y$.
\end{theorem}

 Define on the Hilbert space $\mathcal{H}$ the following \emph{time-dependent inner product}
 \begin{equation*}
 \begin{split}
\langle(u,v,w),(\tilde{u},\tilde{v},\tilde{w})\rangle_{t}&=\int_0^1 \Big(v(x)\tilde{v}(x)+a(x)u'(x)\tilde{u}'(x)\Big)dx\\
&\quad+a(1)\beta u(1)\tilde{u}(1)+\mu_1 a(1)\tau(t)\int_0^1w(\delta)\tilde{w}(\delta)d\delta
\end{split}
\end{equation*}
for all $(u,v,w),~(\tilde{u},\tilde{v},\tilde{w})\in \mathcal{H}$. It is obvious that the norm $\|\cdot\|_{t}$ induced by the time-dependent inner product is is equivalent to $\|\cdot\|_{\mathcal{H}}$, that is, 
\[C_1\|(u,v,w)\|_{\mathcal{H}}^2\le \|(u,v,\sqrt{\tau(t)}w)\|_{\mathcal{H}}^2\le C_2\|(u,v,w)\|_{\mathcal{H}}^2\]
for some positive constants $C_1=\min\{\tau_0,1\}$ and $C_2=\max\{\tau_1,1\}.$

Let us introduce a new operator
\[\mathcal{\widetilde{A}}(t)=\mathcal{A}(t)-\iota(t)\mathcal{I}\quad \forall t\ge 0,\]
where $\iota(t)=\frac{\sqrt{1+(\tau'(t))^2}}{2\tau(t)}$ and $\mathcal{I}$ represents the identity operator from $\mathcal{H}$ into itself.

\begin{lemma}\label{lem2.2}
For the following Cauchy problem
\begin{equation}
\label{3.5}
\begin{cases}
\frac{d\widetilde{U}(t)}{dt}=\mathcal{\widetilde{A}}(t)\widetilde{U}(t)\quad \forall t>0,\\
\widetilde{U}(0)=U_0,
\end{cases}
\end{equation}
if $U_0\in \mathcal{H}$, then there admits a unique solution $\widetilde{U}(t)$ satisfying $\widetilde{U}(t)\in C([0,\infty),\mathcal{H})$. In addition, if $U_0\in \mathfrak{D}(\mathcal{\widetilde{A}}(0))$, then $\widetilde{U}(t)\in C([0,\infty),\mathfrak{D}(\mathcal{\widetilde{A}}(0)))\cap C^1([0,\infty),\mathcal{H}).$
\end{lemma}
\begin{proof}
We are going to complete this proof by verifying  $(1)-(4)$ in Theorem \ref{thm1} for the operator $\mathcal{\widetilde{A}}(t)$. 

It is obvious to verify that $(1)$ and $(2)$ in Theorem \ref{thm1} for the operator $\mathcal{\widetilde{A}}(t)$ due to \[\mathfrak{D}(\mathcal{\widetilde{A}}(t))=\mathfrak{D}(\mathcal{A}(t))=\mathfrak{D}(\mathcal{A}(0)).\] The rest of this proof is to verify $(3)$ and $(4)$ in Theorem \ref{thm1} for the operator $\mathcal{\widetilde{A}}(t)$ by the following two steps.

\textbf{Step 1. Verification of $(3)$ in Theorem \ref{thm1}.}

It is sufficient to prove that the following three claims hold for the operator $\mathcal{\widetilde{A}}(t)$.

\textbf{Claim 1.} The operator  $\mathcal{\widetilde{A}}(t)$ is dissipative in $\mathcal{H}$.

For any $U=(u,v,w)\in \mathfrak{D}(\mathcal{\widetilde{A}}(t))$, that is, $U=(u,v,w)\in \mathfrak{D}(\mathcal{A}(t))$, it follows from integrations by parts that 
\begin{equation*}
\begin{split}
&\langle \mathcal{A}(t)U,U\rangle_{t}\\
&=\int_0^1 \Big((a(x)u'(x))'v(x)+a(x)u'(x)v'(x)\Big)dx+a(1)\beta v(1)u(1)\\
&\quad+\mu_1 a(1)\int_0^1(\delta\tau'(t)-1)w(\delta)w_\delta(\delta)d\delta\\
&=a(1)v(1)(u'(1)+\beta u(1))+\mu_1 a(1)\int_0^1(\delta\tau'(t)-1)w(\delta)w_\delta(\delta)d\delta\\
&=-\mu_1a(1)v^2(1)-\mu_2a(1)v(1)w(1)-\frac{1}{2}\mu_1 a(1)w^2(1)(1-\tau'(t))\\
&\quad+\frac{1}{2}\mu_1 a(1)v^2(1)-\frac{1}{2}\mu_1 a(1)\tau'(t)\int_0^1w^2(\delta)d\delta.
\end{split}
\end{equation*}
Therefore,
\begin{equation*}
\begin{split}
&\langle \mathcal{\widetilde{A}}(t)U,U\rangle_{t}=\langle \mathcal{A}(t)U,U\rangle_{t}-\iota(t)\langle U,U\rangle_{t}\\
&\le -\mu_1a(1)v^2(1)-\mu_2a(1)v(1)w(1)-\frac{1}{2}\mu_1 a(1)w^2(1)(1-\tau'(t))+\frac{1}{2}\mu_1a(1)v^2(1)\\
&\le \Big[-a(1)\Big(\mu_1-\frac{1}{2}\mu_1-\frac{|\mu_2|}{2\sqrt{1-d}}\Big)\Big]v^2(1)+\Big[a(1)\Big(\frac{|\mu_2|\sqrt{1-d}}{2}-\frac{1}{2}\mu_1(1-d)\Big)\Big]w^2(1)\le 0.
\end{split}
\end{equation*}
Here, we have used Cauchy's inequality, \eqref{1.3} and \eqref{2144}.
Therefore, the operator $\mathcal{\widetilde{A}}(t)$ is dissipative.

\textbf{Claim 2.} $R(I-\mathcal{\widetilde{A}}(t))=\mathcal{H}$, where $R(I-\mathcal{\widetilde{A}}(t))$ is the range of operator $I-\mathcal{\widetilde{A}}(t)$.

We need to show that for any $G=(f,g,h)\in \mathcal{H}$, there exists a solution $U=(u,v,w)\in \mathfrak{D}(\mathcal{A}(t))$ of the equation 
\[(I-\mathcal{A}(t))U=G,\]
that is 
\begin{equation*}
v= u-f, 
\end{equation*}
\begin{equation*}
\begin{split}
u-(a(x)u')'=f+g, 
\end{split}
\end{equation*}
\begin{equation}
\label{3.9}w+\frac{1-\delta\tau'(t)}{\tau(t)}w_\delta=h.
\end{equation}
In what follows, we will discuss this in two cases.

\textbf{ Case 1 }$\tau'(t)=0$.\\
Consider the bilinear form $\mathcal{B}:W_a^1(0,1)\times W_a^1(0,1)\to \mathbb{R}$ given by
\[\mathcal{B}(u,\varphi)=\int_0^1\Big(u\varphi+a(x)u'\varphi'\Big)dx+a(1)u(1)\varphi(1)(\mu_1+\mu_2e^{-\tau(t)}+\beta)\]
and the linear form $\mathcal{L}:W_a^1(0,1)\to \mathbb{R}$ 
\[\mathcal{L}\varphi=\int_0^1(f+g)\varphi dx+a(1)\varphi(1)\Big(\mu_1f(1)+\mu_2e^{-\tau(t)}f(1)-\mu_2e^{-\tau(t)}\tau(t)\int_0^1e^{\delta\tau(t)}h(\delta) d\delta\Big).\]
For $(u,\varphi)\in W_a^1(0,1)\times W_a^1(0,1)$, it follows \eqref{1538} that 
\begin{equation}
\label{1548}
\begin{split}
|\mathcal{B}(u,\varphi)|&\le \|u\|_{1,a}\|\varphi\|_{1,a}+a(1)\max\Big\{2,\frac{1}{a(1)}\Big\}(\mu_1+\mu_2e^{-\tau(t)}+\beta)\|u\|_{1,a}\|\varphi\|_{1,a}\\
&\le \Big[1+a(1)\max\Big\{2,\frac{1}{a(1)}\Big\}(\mu_1+\mu_2+\beta)\Big]\|u\|_{1,a}\|\varphi\|_{1,a},
\end{split}
\end{equation}
thus $\mathcal{B}$ is continuous.  Based on  \eqref{2144} and \eqref{1.3}, one has $\mu_1+\mu_2e^{-\tau(t)}+\beta\ge 0$, thus $\mathcal{B}$ is coercive on $W_a^1(0,1)\times W_a^1(0,1)$.  $\mathcal{L}$ is a continuous linear functional. By the Lax-Milgram theorem, there exists a unique solution $u\in W_a^1(0,1)$ of the variational problem
\begin{equation}
\label{3.10}
\mathcal{B}(u,\varphi)=\mathcal{L}\varphi\quad \forall \varphi\in W_a^1(0,1).
\end{equation}
Let us set $v=u-f$, then $v\in W_a^1(0,1)$. It follows from \eqref{3.9} that 
\begin{equation}
\label{0824}
w(\delta)=e^{-\tau(t)}\Big(v(1)+\tau(t)\int_0^\delta e^{\rho\tau(t)}h(\rho) d\rho\Big).
\end{equation}

Since $C_c^\infty(0,1)\subset W_a^1(0,1)$, one has 
\[\int_0^1\Big(u\varphi+a(x)u'\varphi'\Big)dx=\int_0^1(f+g)\varphi dx\quad \forall \varphi\in C_c^\infty(0,1).\]
Therefore, $u-(a(x)u')'=f+g$ holds in the sense of distributions. This indicates $u\in W_a^2(0,1)$ and 
\[u-(a(x)u')'=f+g\quad a.e. \text{ in }(0,1).\]
Using an integration by parts, one gets
\begin{equation}
\label{0854}
\int_0^1 u\varphi dx+\int_0^1a(x)u'\varphi'dx-a(1)u'(1)\varphi(1)=\int_0^1(f+g)\varphi dx\quad \forall \varphi\in W_a^1(0,1).
\end{equation}
This together with \eqref{3.10} yields
\begin{equation}
\label{0836}
\begin{split}
u'(1)&=\mu_1f(1)+\mu_2e^{-\tau(t)}f(1)-\mu_2e^{-\tau(t)}\tau(t)\int_0^1e^{\delta\tau(t)}h(\delta) d\delta-\mu_1u(1)-\mu_2e^{-\tau(t)}u(1)-\beta u(1)\\
&=-\mu_1(u(1)-f(1))-\mu_2e^{-\tau(t)}\Big(u(1)-f(1)+\tau(t)\int_0^1e^{\delta\tau(t)}h(\delta) d\delta\Big)-\beta u(1).
\end{split}
\end{equation}
Here, we have used $a(1)>0$ and the function $\varphi$ defined by $\varphi(x)=x$ for all $x\in (0,1)$ belongs to $W_a^1(0,1)$. Recall  $v=u-f$ and set $\delta=1$ in \eqref{0824}, then \eqref{0836} yields
\begin{equation}
\label{2114}
\mu_1v(1)+\mu_2w(1)+u'(1)+\beta u(1)=0.
\end{equation}

\textbf{Case 2} $\tau'(t)\neq0$.\\
Consider the bilinear form $\widetilde{\mathcal{B}}:W_a^1(0,1)\times W_a^1(0,1)\to \mathbb{R}$ given by
\[\widetilde{\mathcal{B}}(u,\varphi)=\int_0^1\Big(u\varphi+a(x)u'\varphi'\Big)dx+a(1)u(1)\varphi(1)\Big(\mu_1+\mu_2e^{\frac{\tau(t)}{\tau'(t)}\ln(1-\tau'(t))}+\beta\Big)\]
and the linear form $\tilde{\mathcal{L}}:W_a^1(0,1)\to \mathbb{R}$
\begin{equation*}
\begin{split}
\widetilde{\mathcal{L}}\varphi&=\int_0^1(f+g)\varphi dx+a(1)\varphi(1)\Big(\mu_1f(1)+\mu_2e^{\frac{\tau(t)}{\tau'(t)}\ln(1-\tau'(t))}f(1)\\
&\quad-\mu_2e^{\frac{\tau(t)}{\tau'(t)}\ln(1-\tau'(t))}\tau(t)\int_0^1e^{-\frac{\tau(t)\ln(1-\tau'(t)\delta)}{\tau'(t)}}\frac{h(\delta)}{1-\delta\tau'(t)} d\delta\Big).
\end{split}
\end{equation*}
Similar to \eqref{1548}, one has that $\widetilde{\mathcal{B}}$ is  continuous on $W_a^1(0,1)\times W_a^1(0,1)$.  Based on  \eqref{2144} and \eqref{1.3}, one has $\mu_1+\mu_2e^{\frac{\tau(t)}{\tau'(t)}\ln(1-\tau'(t))}+\beta\ge 0$, thus $\widetilde{\mathcal{B}}$ is coercive on $W_a^1(0,1)\times W_a^1(0,1)$.   $\widetilde{\mathcal{L}}$ is a continuous linear functional. By the Lax-Milgram theorem, there exists a unique solution $u\in W_a^1(0,1)$ of the variational problem
\begin{equation}
\label{d3.10}
\widetilde{\mathcal{B}}(u,\varphi)=\widetilde{\mathcal{L}}\varphi\quad \forall \varphi\in W_a^1(0,1).
\end{equation}
Let us set $v=u-f$, then $v\in W_a^1(0,1)$. It follows from \eqref{3.9} that 
\begin{equation}
\label{d0824}
w(\delta)=e^{\frac{\tau(t)}{\tau'(t)}\ln(1-\delta\tau'(t))}\Big(v(1)+\tau(t)\int_0^\delta e^{-\frac{\tau(t)\ln(1-\tau'(t)\rho)}{\tau'(t)}}\frac{h(\rho)}{1-\rho\tau'(t)} d\rho\Big).
\end{equation}

Obviously, \eqref{0854} still holds, wihch together with \eqref{d3.10}  yields 
\begin{equation}
\label{d0836}
\begin{split}
u'(1)&=-\mu_1(u(1)-f(1))-\beta u(1)\\
&\quad-\mu_2e^{\frac{\tau(t)}{\tau'(t)}\ln(1-\tau'(t))}\Big(u(1)-f(1)+\tau(t)\int_0^1e^{-\frac{\tau(t)\ln(1-\tau'(t)\delta)}{\tau'(t)}}\frac{h(\delta)}{1-\delta\tau'(t)} d\delta\Big).
\end{split}
\end{equation}
Here, we have used $a(1)>0$ and the function $\varphi$ defined by $\varphi(x)=x$ for all $x\in (0,1)$ belongs to $W_a^1(0,1)$. Recall  $v=u-f$ and set $\delta=1$ in \eqref{d0824}, then \eqref{d0836} also yields \eqref{2114}.

Recall the boundedness of $\iota(t)$, Claim 2 is obtained directly.

So far, the operator $\mathcal{\widetilde{A}}$ generates a $C_0$ semigroup of contractions on $\mathcal{H}$ by the Lumer-Phillips theorem \cite[Theorem 4.3]{P1983}.

\textbf{Claim 3.} For any $U=(u,v,w)\in \mathcal{H}$, 
\[\frac{\|U\|_t}{\|U\|_s}\le e^{\frac{d}{2\tau_0}|t-s|}\quad \forall t,~s\in[0,T].\]

 For $0\le s\le  t\le T$, we get
\begin{equation}
\label{3.17}
\begin{split}
\|U\|_t^2-\|U\|_s^2e^{\frac{d}{\tau_0}|t-s|}
&=
\Big(1-e^{\frac{d}{\tau_0}|t-s|}\Big)\Big[\int_0^1 \Big(v^2(x)+a(x)(u'(x))^2\Big)dx+a(1)\beta u^2(1)\Big]\\
&\quad+\Big(\tau(t)-\tau(s)e^{\frac{d}{\tau_0}|t-s|}\Big)\mu_1 a(1)\int_0^1w^2d\delta.
\end{split}
\end{equation}
It is not difficult to find that $\tau(t)$ can be represented as 
\[\tau(t)=\tau(s)+\tau'(\varsigma)(t-s)\quad\forall \varsigma\in (s,t).\]
Then, 
\[\frac{\tau(t)}{\tau(s)}\le 1+\frac{d}{\tau_0}(t-s)\le e^{\frac{d}{\tau_0}|t-s|},\]
which yields $\tau(t)-\tau(s)e^{\frac{d}{\tau_0}|t-s|}\le0.$ Note that $1-e^{\frac{d}{\tau_0}|t-s|}\le0$, together with \eqref{3.17}, we prove this claim by \cite[Proposition 1.1]{K1976}.

In conclusion, the family $\mathcal{\widetilde{A}}=\{\mathcal{\widetilde{A}}(t)|t\in [0,T]\}$ is stable in $\mathcal{H}.$

\textbf{Step 2. Verification of $(3)$ in Theorem \ref{thm1}.}

A direct computation implies 
\begin{equation*}
\frac{d}{dt}\mathcal{A}(t)(u,v,w)=\left(0,0,\frac{\delta[\tau(t)\tau''(t)-(\tau'(t))^2]\tau'(t)}{\tau^2(t)}w_\delta\right)
\end{equation*}
Since $\frac{\delta[\tau(t)\tau''(t)-(\tau'(t))^2]\tau'(t)}{\tau^2(t)}$ is bounded, we have $\frac{d}{dt}\mathcal{A}(t)$ is bounded. $\iota'(t)$ is bounded, thus 
\[\frac{d}{dt}\mathcal{\widetilde{A}}(t)=\frac{d}{dt}\mathcal{A}(t)-\iota'(t)\mathcal{I}\in L_*^\infty([0,T],B(\mathfrak{D}(\mathcal{A}(0)),\mathcal{H})).\]

This completes this proof.
\end{proof}

\begin{theorem}\label{thm2.3}
For any initial data $U_0\in \mathcal{H}$, there admits a unique solution $U(t)$ satisfying $U(t)\in C([0,\infty),\mathcal{H})$ for the Cauchy problem \eqref{3.4}. In addition, if $U_0\in \mathfrak{D}(\mathcal{A}(0))$, then $U(t)\in C([0,\infty),\mathfrak{D}(\mathcal{A}(0)))\cap C^1([0,\infty),\mathcal{H}).$
\end{theorem}
\begin{proof}
If $U_0\in \mathfrak{D}(\mathcal{\widetilde{A}}(0))$, then there exists a unique solution $\widetilde{U}(t)\in C([0,\infty),\mathfrak{D}(\mathcal{\widetilde{A}}(0)))\cap C^1([0,\infty),\mathcal{H})$ for problem \eqref{3.5} by Lemma \ref{lem2.2}. Thus, 
\[U(t)=e^{\int_0^t\iota(s)ds}\widetilde{U}(t)\]
is a solution of problem \eqref{3.4}.

In fact, differentiate the above equality, then
\begin{equation*}
\begin{split}
\frac{d}{dt}U(t)&=\iota(t)e^{\int_0^t\iota(s)ds}\widetilde{U}(t)+e^{\int_0^t\iota(s)ds}\frac{d}{dt}\widetilde{U}(t)\\
&=e^{\int_0^t\iota(s)ds}[\iota(t)+\mathcal{\widetilde{A}}(t)]\widetilde{U}(t)\\
&=\mathcal{A}(t)e^{\int_0^t\iota(s)ds}\widetilde{U}(t)=\mathcal{A}(t)U(t).
\end{split}
\end{equation*}
This completes the proof. 
\end{proof}

\begin{corollary}
If $(u_0,u_1)\in W_a^1(0,1)\times L^2(0,1)$, then there exists a unique mild solution for problem \eqref{1.1} 
\[u\in C^1([0,\infty);L^2(0,1))\cap C([0,\infty);W_a^1(0,1)).\]
If $(u_0,u_1)\in W_a^2(0,1)\times W_a^1(0,1)$, then there exists a unique classical solution for problem \eqref{1.1} 
\[u\in C^2([0,\infty);L^2(0,1))\cap C^1([0,\infty);W_a^1(0,1))\cap C^2([0,\infty);W_a^2(0,1)).\]
\end{corollary}

\section{Exponential stability}\label{sec4}
Let us define the energy of problem  \eqref{1.1} as
\begin{equation}
\label{4.1}
E_u(t):=\frac12\Big[\int_0^1\Big(u_t^2+a(x)u_x^2\Big)dx+\beta a(1)u^2(t,1)+\mu_1a(1)\tau(t)\int_0^1u_t^2(t-\delta\tau(t),1)d\delta\Big].
\end{equation}

\begin{lemma}\label{lem4.1}
For $(u_0,u_1)\in W_a^2(0,1)\times W_a^1(0,1)$, the energy of problem  \eqref{1.1} is nonincreasing with respect to $t$, and 
\begin{equation}
\label{4.3}
E_u'(t)\le -C_3[a(1)u_t^2(t,1)+a(1)u_t^2(t-\tau(t),1)],
\end{equation}
with \[C_3=\min\Big\{\frac{\mu_1}{2}-\frac{|\mu_2|}{\sqrt{1-d}},\frac{\mu_1}{2} (1-d)-\frac{|\mu_2|}{2}\sqrt{1-d}\Big\}.\]
\end{lemma}
\begin{proof}
Differentiate \eqref{4.1} with respect to $t$, after integrations by parts, one has
\begin{equation}
\label{4.4}
\begin{split}
E_u'(t)&=\int_0^1(u_tu_{tt}+a(x)u_xu_{xt})dx+\beta a(1)u(t,1)u_t(t,1)+\frac{\mu_1}{2}a(1)\tau'(t)\int_0^1u_t^2(t-\delta\tau(t),1)d\delta\\
&\quad+\mu_1 a(1)\tau(t)\int_0^1u_{tt}(t-\delta\tau(t),1)u_t(t-\delta\tau(t),1)(1-\delta\tau'(t))d\delta\\
&=a(1)u_x(t,1)u_t(t,1)+\beta a(1)u(t,1)u_t(t,1)+\frac{\mu_1}{2}a(1)\tau'(t)\int_0^1u_t^2(t-\delta\tau(t),1)d\delta\\
&\quad+\mu_1 a(1)\tau(t)\int_0^1u_{tt}(t-\delta\tau(t),1)u_t(t-\delta\tau(t),1)(1-\delta\tau'(t))d\delta.
\end{split}
\end{equation}
 Note that 
 \[u_t(t-\delta\tau(t),1)=-\frac{1}{\tau(t)}\partial_\delta u(t-\delta\tau(t),1),\]
 \[u_{tt}(t-\delta\tau(t),1)=\frac{1}{\tau^2(t)}\partial_{\delta\delta} u(t-\delta\tau(t),1).\]
 Inserting these into \eqref{4.4}, using suitable integrations by parts and then applying Cauchy's inequality, one obtains
 \begin{equation*}
\begin{split}
E_u'(t)&=a(1)u_x(t,1)u_t(t,1)+\beta a(1)u(t,1)u_t(t,1)+\frac{\mu_1}{2}a(1)\tau'(t)\int_0^1u_t^2(t-\delta\tau(t),1)d\delta\\
&\quad-\frac{\mu_1} {2}a(1)\tau'(t)\int_0^1u_t^2(t-\delta\tau(t),1)d\delta-\frac{\mu_1}{2}a(1)(1-\tau'(t))u_t^2(t-\tau(t),1)+\frac{\mu_1} {2}a(1)u_t^2(t,1)\\
&=-\Big(\mu_1-\frac{\mu_1} {2}\Big)a(1)u_t^2(t,1)-\mu_2 a(1)u_t(t,1)u_t(t-\tau(t),1)-\frac{\mu_1} {2}a(1)(1-\tau'(t))u_t^2(t-\tau(t),1)\\
&\le -\Big(\mu_1-\frac{|\mu_2|}{\sqrt{1-d}}-\frac{\mu_1} {2}\Big)a(1)u_t^2(t,1)-\Big[\frac{\mu_1} {2} (1-d)-\frac{|\mu_2|}{2}\sqrt{1-d}\Big]a(1)u_t^2(t-\tau(t),1).
\end{split}
\end{equation*}
 Here, the boundary condition $\mu_1u_t(t,1)+\mu_2u_t(t-\tau(t),1)+u_x(t,1)+\beta u(t,1)=0$ are used.  This completes this proof due to \eqref{2144}.
\end{proof}

 Let us introduce the Lyapunov functional
 \begin{equation}
 \label{4.6}
\tilde{E}_u(t):=E_u(t)+\varepsilon\Big[\int_0^1\Big(2xu_xu_t+\frac{\mu_a}{2}uu_t\Big)dx+\mu_1 a(1)\tau(t)\int_0^1e^{-2\delta\tau(t)}u_t^2(t-\delta\tau(t),1)d\delta\Big],
\end{equation}
where $\varepsilon$ is a positive and sufficiently small constant that we will choose hereinafter.
\begin{lemma}\label{lem4.2}
 The functional $\tilde{E}_u(t)$ is equivalent to the energy $E_u(t)$, that is, there exist two positive constants $C_4$ and $C_5$ such that
\[C_4E_u(t)\le \tilde{E}_u(t)\le C_5 E_u(t).\]
\end{lemma}
\begin{proof}
Using Cauchy's inquality and \eqref{1.2}(2), it follows that
\begin{equation*}
\Big|\int_0^1 2xu_xu_tdx\Big|\le \frac{1}{a(1)}\int_0^1a(x)u_x^2dx+\int_0^1u_t^2dx.
\end{equation*}
Using Cauchy's inequality again and \eqref{2021}, one gets
\begin{equation*}
\Big|\int_0^1\frac{\mu_a}{2}uu_tdx\Big|\le \frac{\mu_a}{2}u^2(1)+\frac{\mu_a}{4}C_a'\int_0^1a(x)(u'(x))^2dx+\frac{\mu_a}{4}\int_0^2u_t^2dx.
\end{equation*}
Recall the definition of $E_u(t)$ in \eqref{4.1}, we can prove this lemma if we chose $\varepsilon$ small enough such that 
 \[C_4=1-2\varepsilon\max\Big\{1+\frac{\mu_a}{4},\frac{1}{a(1)}+\frac{\mu_a}{4}C_a',\frac{\mu_a}{2\beta a(1)}\Big\}>0,\]
 \[C_5=1+2\varepsilon\max\Big\{1+\frac{\mu_a}{4},\frac{1}{a(1)}+\frac{\mu_a}{4}C_a',\frac{\mu_a}{2\beta a(1)}\Big\}.\]
\end{proof}

\begin{lemma}\label{lem4.3}
For $(u_0,u_1)\in W_a^2(0,1)\times W_a^1(0,1)$, the solution for  problem  \eqref{1.1} satisfies 
\begin{equation}
\label{4.14}
\begin{split}
&\frac{d}{dt}\Big(\int_0^12xu_xu_tdx+\int_0^1\frac{\mu_a}{2}uu_tdx\Big)\\
&\quad\le-\frac{2-\mu_a}{2}\Big[\int_0^1(u_t^2+a(x)u_x^2)dx+\beta a(1)u^2(t,1)\Big]\\
&\qquad+\Big(1+\frac 52a(1)\mu_1^2\Big)u_t^2(t,1)+\frac 52a(1)\mu_2^2u_t^2(t-\tau(t),1)\\
&\qquad+a(1)\Big[\beta(\beta-\mu_a+1)+\Big(2\beta-\frac{\mu_a}{2}\Big)^2\Big]u^2(t,1).
\end{split}
\end{equation}
\end{lemma}
\begin{proof}
Using \eqref{1.1} and suitable integrations by parts, it follows that 
\begin{equation}
\label{4.7}
\begin{split}
\frac{d}{dt}\int_0^12xu_xu_tdx=u_t^2(t,1)-\int_0^1u_t^2dx+a(1)u_x^2(t,1)-\int_0^1(a(x)-xa'(x))u_x^2dx,
\end{split}
\end{equation}
\begin{equation}
\label{4.8}
\begin{split}
\frac{d}{dt}\int_0^1\frac{\mu_a}{2}uu_tdx=\frac{\mu_a}{2} \int_0^1u_t^2dx+\frac{\mu_a}{2}a(1)u_x(t,1)u(t,1)-\frac{\mu_a}{2}\int_0^1a(x)u_x^2dx.
\end{split}
\end{equation}
Here, note that $W_a^2(0,1)\subset V_a^2(0,1)$ and $W_a^1(0,1)\subset V_a^1(0,1)$,  we have used $xu_t^2(t,x)|_{x=0}=0$ and $(xa(x)u_x^2(t,x))|_{x=0}=0$ according to $u\in W_a^2(0,1)$, $u_t\in W_a^1(0,1)$ and Proposition \ref{prop2.1}.

Keeping the boundary condition $\mu_1u_t(t,1)+\mu_2u_t(t-\tau(t),1)+u_x(t,1)+\beta u(t,1)=0$ and $(A+B+C)^2=A^2+B^2+C^2-2AB-2AC-2BC$ in mind, combining \eqref{4.7} with \eqref{4.8}, then
\begin{equation}
\label{4.9}
\begin{split}
&\frac{d}{dt}\Big(\int_0^12xu_xu_tdx+\int_0^1\frac{\mu_a}{2}uu_tdx\Big)\\
&\quad=-(2-\mu_a)\int_0^1\frac12u_t^2dx-\int_0^1\frac12[2(a(x)-xa'(x))+\mu_a]u_x^2dx\\
&\qquad+[1+a(1)\mu_1^2]u_t^2(t,1)+a(1)\mu_2^2u_t^2(t-\tau(t),1)+\Big(\beta-\frac{\mu_a}{2}\Big)a(1)\beta u^2(t,1)\\
&\qquad+2a(1)\mu_1\mu_2u_t(t,1)u_t(t-\tau(t),1)+\Big(2\beta-\frac{\mu_a}{2}\Big)a(1)\mu_1u_t(t,1)u(t,1)\\
&\qquad+\Big(2\beta-\frac{\mu_a}{2}\Big)a(1)\mu_2u_t(t-\tau(t),1)u(t,1)\\
&\quad\le-\frac{2-\mu_a}{2}\Big[\int_0^1(u_t^2+a(x)u_x^2)dx+\beta a(1)u^2(t,1)\Big]\\
&\qquad+[1+a(1)\mu_1^2]u_t^2(t,1)+a(1)\mu_2^2u_t^2(t-\tau(t),1)+a(1)\beta(\beta-\mu_a+1)u^2(t,1)\\
&\qquad+2a(1)\mu_1\mu_2u_t(t,1)u_t(t-\tau(t),1)+\Big(2\beta-\frac{\mu_a}{2}\Big)a(1)\mu_1u_t(t,1)u(t,1)\\
&\qquad+\Big(2\beta-\frac{\mu_a}{2}\Big)a(1)\mu_2u_t(t-\tau(t),1)u(t,1).
\end{split}
\end{equation}
Here, we have used $(2-\mu_a)a(x)\le 2[a(x)-xa'(x)]+a(x)\mu_a$.

By Cauchy's inequality, it follows that 
\begin{equation}
\label{4.10}
\Big(2\beta-\frac{\mu_a}{2}\Big)a(1)\mu_1u_t(t,1)u(t,1)\le \frac12a(1)\mu_1^2u_t^2(t,1)+\frac12\Big(2\beta-\frac{\mu_a}{2}\Big)^2a(1)u^2(t,1),
\end{equation}
\begin{equation}
\label{4.11}
\Big(2\beta-\frac{\mu_a}{2}\Big)a(1)\mu_2u_t(t-\tau(t),1)u(t,1)\le \frac12a(1)\mu_2^2u_t^2(t-\tau(t),1)+\frac12\Big(2\beta-\frac{\mu_a}{2}\Big)^2a(1)u^2(t,1),
\end{equation}
\begin{equation}
\label{4.12}
2a(1)\mu_1\mu_2u_t(t,1)u_t(t-\tau(t),1)\le a(1)\mu_1^2u_t^2(t,1)+a(1)\mu_2^2u_t^2(t-\tau(t),1).
\end{equation}
Inserting \eqref{4.10}-\eqref{4.12} into \eqref{4.9}, we get \eqref{4.14}.
\end{proof}

\begin{lemma}\label{lem4.4}
For $(u_0,u_1)\in W_a^2(0,1)\times W_a^1(0,1)$, one has
 \begin{equation*}
\begin{split}
&\frac{d}{dt}\Big(\mu_1 a(1)\tau(t)\int_0^1e^{-2\delta\tau(t)}u_t^2(t-\delta\tau(t),1)\Big)\\
&\quad\le -2\mu_1 a(1)\tau(t)\int_0^1e^{-2\delta\tau(t)}u_t^2(t-\delta\tau(t),1)+\mu_1 a(1)u_t^2(t,1)d\delta.
\end{split}
\end{equation*}
\end{lemma}
\begin{proof}
Following the proof of \cite[Lemma 3.2.]{N2011}, we can prove this lemma  by changing $\mathcal{E}(t)$ in \cite[Lemma 3.2.]{N2011} to $\mu_1 a(1)\tau(t)\int_0^1e^{-2\delta\tau(t)}u_t^2(t-\delta\tau(t),1).$ Here, we omit the details.
\end{proof}

\begin{theorem}\label{thm4.5}
Let  \eqref{1.2} be fulfilled. For $(u_0,u_1)\in W_a^1(0,1)\times L^2(0,1)$, the solution for problem \eqref{1.1} satisfies the uniform exponential decay
\begin{equation}
\label{4.15}
E_u(t)\le E_u(0)e^{1-\frac{t}{\tilde{M}}}, \quad \forall t\in[\tilde{M},\infty),
\end{equation}
where $\tilde{M}>0$ is given in \eqref{2104} and is independent of $(u_0,u_1)$.
\end{theorem}
\begin{proof}
Let $(u_0,u_1)\in W_a^2(0,1)\times W_a^1(0,1)$. For the solution $u$, combining \eqref{4.6} with Lemmas \ref{lem4.1}, \ref{lem4.3} and \ref{lem4.4}, it follows that 
\begin{equation}
\label{4.16}
\begin{split}
\frac{d}{dt}\tilde{E}_u(t)&\le -C_4[a(1)u_t^2(t,1)+a(1)u_t^2(t-\tau(t),1)]\\
&\quad -\varepsilon\frac{2-\mu_a}{2}\Big[\int_0^1(u_t^2+a(x)u_x^2)dx+\beta a(1)u^2(t,1)\Big]+\varepsilon\Big(1+\frac 52a(1)\mu_1^2\Big)u_t^2(t,1)\\
&\quad+\varepsilon\frac 52a(1)\mu_2^2u_t^2(t-\tau(t),1)+\varepsilon a(1)\Big[\beta(\beta-\mu_a+1)+\Big(2\beta-\frac{\mu_a}{2}\Big)^2\Big]u^2(t,1)\\
&\quad -\varepsilon2\mu_1 a(1)\tau(t)\int_0^1e^{-2\delta\tau(t)}u_t^2(t-\delta\tau(t),1)d\delta+\varepsilon\mu_1 a(1)u_t^2(t,1)\\
&\le -C_6(u_t^2(t,1)+u_t^2(t-\tau(t),1))-\varepsilon\min\{2-\mu_a,e^{-2\tau_1}\}E_u(t)\\
&\quad +\varepsilon a(1)\Big[\beta(\beta-\mu_a+1)+\Big(2\beta-\frac{\mu_a}{2}\Big)^2\Big]u^2(t,1)\\
&\le -\varepsilon\min\{2-\mu_a,e^{-2\tau_1}\}E_u(t)+\varepsilon C_7a(1)u^2(t,1),
\end{split}
\end{equation}
where $C_6$ and $C_7$ are positive constants.

Integrating \eqref{4.16} from $S$ to $T$ with respect to $t$, one arrives at
\begin{equation}
\label{4.17}
\begin{split}
\varepsilon\min\{2-\mu_a,e^{-2\tau_1}\}\int_S^TE_u(t)dt\le \tilde{E}_u(S)-\tilde{E}_u(T)+\varepsilon C_7\int_S^Ta(1)u^2(t,1)dt.
\end{split}
\end{equation}

We are now in a position to estimate $\int_S^Ta(1)u^2(t,1)dt$. Set $\lambda=u(t,1)$ and denote
by $z$ the solution of the degenerate elliptic problem \eqref{1736}. Multiplying the first equation in \eqref{1.1} by $z$, integrating the resulting equation over $(S,T)\times(0,1)$, and then using integrations by parts, it follows that 
\begin{equation}
\label{4.18}
\begin{split}
\int_S^Ta(1)u^2(t,1)dt&=\int_S^T\int_0^1u_tz_tdxdt-\int_0^1u_tzdx\Big|_S^T\\
&\qquad-\int_S^Ta(1)[\mu_1u_t(t,1)+\mu_2u_t(t-\tau(t),1)]z(t,1)dt.
\end{split}
\end{equation}
It follows from H\"older's inequality, Young's inequality with $\eta>0$ and \eqref{4.3} that 
\begin{equation}
\label{4.19}
\begin{split}
\int_S^T\int_0^1u_tz_tdxdt&\le \int_S^T \frac\eta 2\|u_t\|_{L^2(0,1)}^2dt+\int_S^T\frac{1}{2\eta}\|z_t\|_{L^2(0,1)}^2dt\\
&\le \eta \int_S^T E_u(t)dt+\frac{a(1)}{\beta \alpha_a}\frac{1}{2\eta}\int_S^Tu_t^2(t,1)dt
\\
&\le \eta \int_S^T E_u(t)dt+\frac{1}{\beta \alpha_a}\frac{1}{2\eta}\frac{1}{C_3}[E_u(S)-E_u(T)],
\end{split}
\end{equation}
\begin{equation}
\label{4.20}
\Big|\int_0^1u_tzdx\Big|\le \frac{1}{\beta\sqrt{\alpha_a}}\Big(\int_0^1\frac{u_t^2}{2}dx+\frac{\beta a(1)}{2}u^2(t,1)\Big)\le \frac{1}{\beta\sqrt{\alpha_a}}E_u(t).
\end{equation}
Here, we have used 
\[\|z_t\|_{L^2(0,1)}^2\le \frac{a(1)}{\beta \alpha_a}u_t^2(t,1),\]
\begin{equation}
\label{1030}
z^2(t,1)\le \frac{1}{\beta^2}u^2(t,1)\le \frac{2}{\beta^3a(1)}E_u(t)
\end{equation}
based on \eqref{1729}.

Using Cauchy's inequality, \eqref{1030} and \eqref{4.3}, one has
\begin{equation}
\label{4.21}
\begin{split}
&\Big|\int_S^Ta(1)[\mu_1u_t(t,1)+\mu_2u_t(t-\tau(t),1)]z(t,1)dt\Big|\\
&\quad\le\frac{1}{2\eta}\int_S^T [a(1)u_t^2(t,1)+a(1)u_t^{2}(t-\tau(t),1)]dt+\frac{2\eta}{\beta^3}\int_S^T E_u(t)dt\\
&\quad\le\frac{1}{2\eta C_3}\int_S^T -E_u'(t)dt+\frac{2\eta}{\beta^3}\int_S^T E_u(t)dt
\\
&\quad\le\frac{1}{2\eta C_3}(E_u(S)-E_u(T))+\frac{2\eta}{\beta^3}\int_S^T E_u(t)dt.
\end{split}
\end{equation}
Combining \eqref{4.19}-\eqref{4.21} with \eqref{4.18}, then
\begin{equation*}
\begin{split}
\int_S^Ta(1)u^2(t,1)dt&\le \Big(1+\frac{2}{\beta^3}\Big)\eta\int_S^T E_u(t)dt+\frac{1}{\beta \alpha_a}\frac{1}{2\eta}\frac{1}{C_3}[E(S)-E_u(T)]\\
&\qquad+\frac{1}{\beta\sqrt{\alpha_a}}[E(S)+E_u(T)]+\frac{1}{2\eta C_3}[E(S)-E_u(T)],
\end{split}
\end{equation*}
which together with \eqref{4.17} yields 
\begin{equation}
\label{4.23}
\begin{split}
&\varepsilon\min\{2-\mu_a,e^{-2\tau_1}\}\int_S^TE_u(t)dt\\
&\quad\le \tilde{E}_u(S)-\tilde{E}_u(T)+\varepsilon C_7\Big(1+\frac{2}{\beta^3}\Big)\eta\int_S^T E_u(t)dt+\varepsilon C_7\frac{1}{\beta \alpha_a}\frac{1}{2\eta}\frac{1}{C_3}[E_u(S)-E_u(T)]\\
&\qquad+\frac{\varepsilon C_7}{\beta\sqrt{\alpha_a}}[E_u(S)+E_u(T)]+\frac{\varepsilon C_7}{2\eta C_3}[E_u(S)-E_u(T)].
\end{split}
\end{equation}
Recall Lemma \ref{lem4.2}, let us choose \[\eta=\frac{\min\{2-\mu_a,e^{-2\tau_1}\}}{2C_7\Big(1+\frac{2}{\beta^3}\Big)},\]
then it follows from \eqref{4.23} that 
\begin{equation*}
\int_S^TE_u(t)dt\le \tilde{M}E_u(S),
\end{equation*}
where
\begin{equation}
\label{2104}
\begin{split}
\tilde{M}&=\frac{2}{\varepsilon\min\{2-\mu_a,e^{-2\tau_1}\}}\Big[C_5+\varepsilon \frac{1}{\beta \alpha_a}\frac{C_7^2\Big(1+\frac{2}{\beta^3}\Big)}{\min\{2-\mu_a,e^{-2\tau_1}\}}\frac{1}{C_3}\\
&\qquad+\frac{2\varepsilon C_7}{\beta\sqrt{\alpha_a}}+\frac{C_7^2\Big(1+\frac{2}{\beta^3}\Big)}{\min\{2-\mu_a,e^{-2\tau_1}\}} \frac{\varepsilon}{C_3}\Big]>0.
\end{split}
\end{equation}
It is direct to obtain \eqref{4.15} via Theorem \ref{thmin}.

For $(u_0,u_1)\in W_a^1(0,1)\times L^2(0,1)$, let us consider a sequence $\{(u_0^n,u_1^n)\}_{n\in \mathbb{N}}\in W_a^2(0,1)\times W_a^1(0,1)$  that approximate to $(u_0,u_1)\in W_a^1(0,1)\times L^2(0,1)$. Let $u^n$ be the solution associated to $\{(u_0^n,u_1^n)\}_{n\in \mathbb{N}}\in W_a^2(0,1)\times W_a^1(0,1)$. Then the desired result can be extend to any solution corresponding to $(u_0,u_1)\in W_a^1(0,1)\times L^2(0,1)$ by the approximation argument.

\end{proof}

\section*{Acknowledgements} 
M. Liao was supported by the National Natural Science Foundation of China(12401290) and the Natural Science Foundation of Jiangsu Province(BK20230946) and the Fundamental Research Funds for the Central Universities(B230201033, B240201090).

\section*{Competing Interests}
The authors declare that they have no competing interests.

\section*{Data Availability}
Data sharing is not applicable to this article as no new data were created or analyzed in this study.

\end{document}